\documentclass[12pt]{amsart}
\pdfoutput=1

\usepackage{enumerate}
\usepackage{amsmath,amsthm,amssymb}
\usepackage{color}
\usepackage{url}

\usepackage{graphicx}
\usepackage{tikz}
\usetikzlibrary{calc}
\usepackage{wrapfig}

\newcommand{\projectivespace}[0]{\mathbb{P}^N}
\newcommand{\grassmannian}[0]{\mathbb{G}}
\newcommand{\degree}[0]{d_X}
\newcommand{\chow}[0]{\mathrm{Ch}_X}
\newcommand{\hurwitz}[0]{\mathrm{Hu}_X}
\newcommand{\torus}[0]{\mathbb{H}}
\newcommand{\torustoric}[0]{\mathbb{T}}
\newcommand{\euclidspacen}[0]{\mathbb{R}^n}
\newcommand{\euclidspaceN}[0]{\mathbb{R}^{N+1}}
\newcommand{\chowpolytope}[0]{\mathcal{W}(\mathrm{Ch}_X)}
\newcommand{\hurwitzpolytope}[0]{\mathcal{W}(\mathrm{Hu}_X)}
\newcommand{\volume}[0]{\mathrm{Vol}_\mathbb{Z}}
\newcommand{\secondarypolytope}[0]{\mathrm{SecPoly}}
\newcommand{\chowdegree}[0]{\mathrm{deg}(\mathrm{Ch}_X)}
\newcommand{\hurwitzdegree}[0]{\mathrm{deg}(\mathrm{Hu}_X)}

\newcommand{\coisotropic}[0]{\mathrm{CH}}

\newcommand{\ddbar}[0]{\partial\bar{\partial}}
\newcommand{\kahlerpotentialspace}[0]{\mathcal{P}(X,\Omega)}

\newcommand{\fzero}[0]{\mathcal{F}^0_{\omega_0}}
\newcommand{\kenergy}[0]{\mathcal{M}_{\omega_0}}
\newcommand{\scalar}[0]{\mathrm{Scal}}
\newcommand{\averagescalar}[0]{\overline{S}}
\newcommand{\fubinistudy}[0]{\omega_{\mathrm{FS}}}

\newcommand{\glg}[0]{\mathrm{GL}(N+1,\mathbb{C})}
\newcommand{\latticecharacter}[0]{M_{\mathbb{Z}}}
\newcommand{\latticecharacterreal}[0]{M_{\mathbb{R}}}
\newcommand{\latticecharacterdual}[0]{N_{\mathbb{Z}}}
\newcommand{\paulfutaki}[0]{\mathrm{FP}}
\newcommand{\nakenergy}[0]{\mathcal{M}^{\mathrm{NA}}}
\newcommand{\testconfigurationspace}[0]{\mathcal{X}_\lambda}
\newcommand{\testconfigurationbundle}[0]{\mathcal{L}_\lambda}
\newcommand{\ctimes}[0]{\mathbb{C}^\times}
\newcommand{\polytopetc}[0]{\bar{Q}}
\newcommand{\toricdf}[0]{\mathcal{L}}
\newcommand{\chowfutaki}[0]{\mu_{\mathrm{Ch}}}
\newcommand{\toricaubin}[0]{\mathcal{I}}
\newcommand{\namaenergy}[0]{\mathcal{E}^{\mathrm{NA}}}
\newcommand{\df}[0]{\mathrm{DF}}

\newcommand{\pointheight}[0]{\widetilde{\omega}}

\newcommand{\Afunction}[0]{\mathbb{R}^A}
\newcommand{\functionlambda}[0]{g_\lambda}
\newcommand{\triangulationlambda}[0]{T_\lambda}
\newcommand{\conet}[0]{\mathrm{C}(T)}
\newcommand{\normalcone}[0]{\mathrm{NC}}

\newcommand{\interior}[0]{\mathrm{Int}}

\newtheorem{theorem}{Theorem}[section]

\newtheorem{corollary}[theorem]{Corollary}
\newtheorem{proposition}[theorem]{Proposition}

\theoremstyle{definition}

\newtheorem{definition}[theorem]{Definition}

\newtheorem{remark}[theorem]{Remark}

\begin{document}

\title{
	Weight polytopes and energy functionals of toric varieties
	}
\author{Yuji Sano}
\address{
Department of Applied Mathematics
    Fukuoka University
    8-19-1 Nanakuma, Jonan-ku, Fukuoka 814-0180, JAPAN
}

\email{
sanoyuji@fukuoka-u.ac.jp
}
\date{}

\thanks{
The author is supported by JSPS KAKENHI Grant Number 22K03325 and by research funds from Fukuoka University (Grant Number 225001-000).}

\begin{abstract}
We prove that the weight polytope of the Hurwitz form of a polarized smooth toric variety coincides with the convex hull of the characteristic vectors introduced in \cite{os22} with respect to all regular triangulations of the momentum polytope.
Our proof relies on the combination of the two slope formulas of $K$-energy \cite{paul12, bhj2019} in the toric setting.

\end{abstract}

\maketitle

\section{Introduction} \label{sec:introduction}

Let $(X,L)$ be an $n$-dimensional smooth polarized variety with very ample line bundle $L$.
The Kodaira embedding theorem implies the embedding 
\begin{equation}\label{eq:embedding}
	\iota:
	X \hookrightarrow \projectivespace
	\simeq \mathbb{P}(H^0(X,L)^*)
\end{equation}
with respect to a basis of $H^0(X, L)$.
We denote the image $\iota(X)$ by $X$ if it does not create any confusion.
Through this note, we assume that $X$ is irreducible and linearly normal and the degree $\degree$ of $X$ is greater than or equal to two.

Let us recall the Hurwitz form of $X$ introduced in \cite{sturmfels17}.
Let $\grassmannian(k, \projectivespace)$ be the Grassmannian of all $k$-dimensional linear subspaces in the $N$-dimensional projective space $\projectivespace$.
For an integer $N-n-1\le k \le N-1$, the subvariety of $\grassmannian(k,\projectivespace)$ defined by the Zariski closure of
\begin{equation*}\label{eq:coisotropic}
	\coisotropic_k(X):=
	\{
		L \in \grassmannian(k,\projectivespace)
		\mid\,
		L 
		\mbox{ intersects non-transversally }
		X
	\}
\end{equation*}
is called \textit{the $($$k$-th$)$ higher associated hypersurface} in \cite{gkz94} (if it has codimension one) or \textit{the coisotropic variety} in \cite{kohn21}.
If $k=N-n-1$ then $\coisotropic_{N-n-1}(X)$ is a hypersurface of $\grassmannian(N-n-1,\projectivespace)$ and its defining polynomial $\chow$ is called \textit{the Chow form} of $X$.
If $k=N-n$ then $\coisotropic_{N-n} (X) $ is a hypersurface of $\grassmannian(N-n, \projectivespace)$ and its defining polynomial $\hurwitz$ is called \textit{the Hurwitz form} of $X$ in \cite{sturmfels17}.
Each polynomial is an irreducible element in the coordinate ring of the corresponding Grassmannian.
\begin{remark}\label{rem:huriwitzhyperdiscriminant}
	By the Cayley trick, the Hurwitz form can be written as the discriminant. 
	More precisely, the Hurwitz form is equal to \textit{the hyperdiscriminant}, i.e., the discriminant of the Segre embedding
$$
	X\times \mathbb{P}^{n-1} 
	\hookrightarrow
	\mathbb{P}^{n(N+1)-1}
$$
	in the Pl\"ucker coordinates.
	See \cite{sturmfels17, kohn21} for the proof.
\end{remark}

Take a maximal torus $\torus$ in $\glg$ and consider the natural action of $\torus$ on $\projectivespace$.
This action is extended to the action on the coordinate ring of $\grassmannian(k,\projectivespace)$ in a natural way.
We call the weight polytopes of the Chow form and the Hurwitz form of $X=\iota(X)$ with respect to the $\torus$-action \textit{the Chow polytope} $\chowpolytope$ and \textit{the Hurwitz polytope} $\hurwitzpolytope$ respectively.
By definition, both $\chowpolytope$ and $\hurwitzpolytope$ are integral polytopes in $\euclidspaceN$.

Now, we consider the case where $X$ is a toric variety, i.e., an algebraic normal variety with an effective action of $\torustoric:=(\mathbb{C}^\times)^n$.
By the general theory of toric geometry, the polarization by a given $\torustoric$-equivariant very ample line bundle $L$ determines the momentum polytope $Q$, which is an integral Delzant polytope in $\euclidspacen$. 
We denote the point configuration consisting of the lattice points contained in $Q$ and its boundary $\partial Q$ by
$$
	A= \{\omega_0,\ldots, \omega_N\}.
$$
Since each point in $A$ corresponds to an element of $H^0(X,L)$, the cardinality of $A$ is equal to $N+1$.
We take the standard torus $\torus$ in $\glg$ so that $$
	\rho(\torustoric)\subset 
	\torus=
	\bigg\{
		\begin{pmatrix}
			t_0 &
		\\
			& \ddots &
		\\
			&& t_N
		\end{pmatrix}
		\bigg|\,
		t_i \in \ctimes
	\bigg\}
	\subset \glg
$$
where $\rho:\torustoric\to \glg$ is a rational representation.
We denote $\rho(\torustoric)$ by $\torustoric$ if it does not create any confusion.

Let us recall the characteristic vectors associated with the weight polytopes of $X$.
Let $T$ be a triangulation of the pair $(Q,A)$, i.e., $T$ is a triangulation of $Q$ and any vertex of any simplex in $T$ is contained in $A$.
For $0\le k \le n$, we denote by $\Sigma_T(k)$ the set of $k$-dimensional simplices contained in $T$.
For each lattice point $\omega_i\in A$, we define 
$$
	\eta_{T,n} (\omega_i):= \sum_{\omega_i\prec \sigma \in \Sigma_T(n)} \volume(\sigma).
$$
In the above, $\sigma$ runs through the set of the simplices in $\Sigma_T(n)$ containing $\omega_i$.
Note that the volume $\volume(\sigma)$ is normalized so that the volume of the standard simplex is equal to one in the affine space of minimal dimension where $\sigma$ is included.
We define \textit{the GKZ vector} with respect to $T$ by
$$
	\eta_T:= (\eta_{T,n}(\omega_0),\ldots,\eta_{T,n}(\omega_{N})).
$$
We call that a subdivision $T$ (not necessarily triangulation) is \textit{regular} when there exists a convex piecewise-linear function $g$ on $Q$ such that the vertical projection of the graph of $g$ to the domain $Q$ induces a subdivision $T$ of $Q$.
We call the convex hull of the GKZ vectors $\eta_T$ of all regular triangulations \textit{the secondary polytope}.
We denote it by $\secondarypolytope(X)$.
The following result is fundamental to the characterization of the Chow polytope of a toric variety.
\begin{theorem}[\cite{ksz92},\cite{gkz94}]\label{thm:chowpolytope}
	The Chow polytope $\chowpolytope$ of a polarized toric variety $(X,L)$ coincides with the secondary polytope $\secondarypolytope(X)$.
\end{theorem}
Ogusu and the author \cite{os22} introduce the characteristic vector for the Hurwitz polytope as an analogue of the GKZ vector.
For a simplex $\sigma$ in $\Sigma_T(n-1)$, we call it \textit{massive} if and only if $\sigma$ is contained in some facet of $Q$.
For each $\omega_i$, we define 
$$
	\eta_{T,n-1} (\omega_i):= \sum_{\omega_i\prec \sigma \in \Sigma_{T}(n-1)} \volume(\sigma).
$$
In the above, $\sigma$ runs through the set of the massive simplices in $\Sigma_T(n-1)$ containing $\omega_i$.
\begin{definition}[\cite{os22}]
For a triangulation of $(Q,A)$, the vector
$$
	\xi_T:= (\xi_T(\omega_0),\ldots,\xi_T(\omega_N))
$$
where
$$
	\xi_T(\omega_i):= n \eta_{T,n}(\omega_i)
	-
	\eta_{T,n-1}(\omega_i)
$$	
is called \textit{the Hurwitz vector} with respect to $T$.
\end{definition}
The main theorem of this note is as follows.
\begin{theorem}\label{thm:main}
	The Hurwitz polytope $\hurwitzpolytope$ of a smooth polarized toric variety $(X,L)$ coincides with the convex hull of the Hurwitz vectors $\xi_T$ with respect to all regular triangulations $T$ of $(Q,A)$.
\end{theorem}

In \cite{os22}, Ogusu and the author attempted to compute the Hurwitz polytope by applying the Gelfand-Kapranov-Zelevinsky (GKZ) theory \cite{gkz94} to the hyperdiscriminant polytope.
Then they achieved a partial result of Theorem \ref{thm:main} in dimension two.
In the proof of Theorem \ref{thm:main}, we employ the slope formulas of the energy functionals in K\"ahler geometry.

The organization of this note is as follows.
In Section \ref{sec:kenergy}, we recall the slope formulas of $K$-energy functionals in K\"ahler geometry and see their coincidence.
In Section \ref{sec:proof}, we give proofs to Theorem \ref{thm:chowpolytope} and Theorem \ref{thm:main}. 
First, we describe the functional $\toricdf$ introduced in \cite{donaldson02} in terms of the vertices of the Hurwitz polytope and the Chow polytope (Proposition \ref{prop:hurwitzweightandtoricdf}).
We also see a similar result on the Aubin functional (Proposition \ref{prop:chowweightandtoricaubin}).
Second, we recover Theorem \ref{thm:chowpolytope} by using Proposition \ref{prop:chowweightandtoricaubin}.
Third, we prove Theorem \ref{thm:main} by usign Proposition \ref{prop:hurwitzweightandtoricdf} in the same way.
The proofs provided here would be an approach from K\"ahler geometry to the study of the weight polytopes of the coisotropic hypersurfaces.

\section{Slope formulas of energy functionals}\label{sec:kenergy}
In this section, we recall some results in K\"ahler geometry needed for the proofs of Theorem \ref{thm:chowpolytope} and Theorem \ref{thm:main}.
Through this section, let $X=\iota(X)$ be an embedded smooth polarized variety by (\ref{eq:embedding}).
Let $\Omega:=c_1(L)$ be the K\"ahler class of $X$.

\subsection{Energy functionals}
Take a reference K\"ahler form $\omega_0\in \Omega$.
We denote the volume of $X$ with respect to the volume form $\omega_0^n$ by $V$.
Let 
$$
	\kahlerpotentialspace:
	=
	\{
		\varphi \in C^\infty(X)_\mathbb{R}
		\mid\,
		\omega_\varphi := 
		\omega_0+(\sqrt{-1}/2\pi) \varphi >0
	\}
$$
be the space of K\"ahler potentials of K\"ahler forms in $\Omega$.
For any K\"ahler form $\omega_{\varphi_1}$, take a path $\varphi_t$ in $\kahlerpotentialspace$ connecting $\omega_0$ to $\omega_{\varphi_1}$.

We define the following functionals on $\kahlerpotentialspace$ by
\begin{eqnarray*}
		\fzero(\varphi)
	&=&
		\frac{1}{V}
		\int^1_0dt
		\int_X 
			\dot{\varphi_t} 
			\omega_{\varphi_t}^n,
	\\
		\kenergy(\varphi)
	&=&
		-
		\frac{1}{V}
		\int^1_0dt
		\int_X
			\dot{\varphi_t}
			(\scalar(\omega_t)-\averagescalar)
			\omega_t^n.
\end{eqnarray*}
In the above, $\scalar(\omega)$ denotes the scalar curvature of $\omega$ and $\averagescalar$ denotes the average of the scalar curvature
$$
	\averagescalar:= \frac{1}{V}\int_X \scalar(\omega)\omega^n
$$
that is independent of the choice of $\omega$.
The functional $\fzero$ is often known as \textit{the Aubin functional} and the functional $\kenergy$ is called \textit{$K$-energy} introduced in \cite{mabuchi86}.
Note that they are independent of the choice of the path $\varphi_t$ and that
$$
	\fzero(\varphi+C)=\fzero(\varphi)+C,
	\quad
	\kenergy(\varphi+C)=\kenergy(\varphi)
$$ 
for any constant $C$.

\subsection{Paul's formula}
Let us recall the formula of $K$-energy on the space of the pull-backed Fubini-Study metrics given by Paul \cite{paul12}.

Let the reference form 
$
	\omega_0=\iota^*\fubinistudy 
	\in \Omega
$
be the pull-backed Fubini-Study form on $\projectivespace$ by the embedding (\ref{eq:embedding}) with respect to a fixed basis of $H^0(X,L)$.
For $\lambda\in\glg$, let
$$
	\lambda^*\omega_0 =
	\omega_0 + (\sqrt{-1}/2\pi)
	\ddbar
	\varphi_\lambda.
$$
The potential $\varphi_\lambda$ is unique up to constant.

For the embedded smooth variety $X=\iota (X)$, we denote the Chow form and the Hurwitz form by $\chow$ and $\hurwitz$ respectively.
We denote the degree of $\chow$ and $\hurwitz$ by $\chowdegree$ and $\hurwitzdegree$ respectively.
Then the followings are known:
\begin{eqnarray}
	\label{eq:chowdegree}
		\chowdegree
	&=&
		\degree,
	\\
	\label{eq:hurwitzdegree}
		\hurwitzdegree
	&=&
		(n+1)\degree - \frac{\degree \averagescalar}{n}.	
\end{eqnarray}
We refer to Proposition 5.7 \cite{paul12} for the proof of (\ref{eq:hurwitzdegree}).
Remark that both $\chowdegree$ and $\hurwitzdegree$ in this note are written in the Pl\"ucker coordinates, whereas the ones are written in the Stiefel coordinates in \cite{paul12} (see  \cite{paul21}).

We denote the action of $\glg$ on the coordinate ring of $\grassmannian(k,\projectivespace)$ by
$
	\lambda\cdot F$
where $\lambda\in\glg$ and $F$ is an element of the coordinate ring of $\grassmannian(k,\projectivespace)$.
  
\begin{theorem}[Theorem A \cite{paul12}]\label{thm:paul}
	For any $\lambda\in\glg$,
	\begin{eqnarray}
		\nonumber
			 (n+1)V^2
			\kenergy(\varphi_\lambda)
		&=&
			(n+1)
			\chowdegree 
			\log
			\frac{\|\lambda\cdot \hurwitz\|^2}
				{\|\hurwitz\|^2}
		\\
		\label{eq:paul}
		&&
			\qquad\qquad
			-
			n
			\hurwitzdegree
			\log
			\frac{\|\lambda\cdot \chow\|^2}
				{\|\chow\|^2}.		
	\end{eqnarray}
\end{theorem}
\noindent
We give four remarks on the above theorem.
First, the difference in the choice of the coordinates on the Grassmannians affects the factors of the first and second terms in the right hand of (\ref{eq:paul}) compared to (1.1) in \cite{paul12}.
Second, the formula (\ref{eq:paul}) is proved by using the hyperdiscriminant of $X$ instead of the Hurwitz form originally in \cite{paul12}.
The hyperdiscriminant plays a role in the proof.
However, we replace it with the Hurwitz form for our purpose.
This replacement does make no difference due to Remark \ref{rem:huriwitzhyperdiscriminant}.
Third, the norms that appeared in Theorem A \cite{paul12} are introduced originally in the pioneering work \cite{tian94} of Tian.
These norms are described more explicitly in \cite{paul21}.
The fact that we should keep in mind is that the norm in (\ref{eq:paul}) are conformally equivalent to the standard norms because the dimension of the spaces we consider are finite.
Fourth, the group appeared in \cite{paul12, paul21} is $\mathrm{SL}(N+1,\mathbb{C})$, whereas the one we consider is $\glg$.
This does make no difference on (\ref{eq:paul}) because the scaling of $\lambda$ does not affect both sides of (\ref{eq:paul}).

Theorem \ref{thm:paul} implies the asymptotic expansion of $K$-energy.
Let $\latticecharacter \simeq \mathbb{Z}^{N+1}$ be the rank $(N+1)$ lattice of the rational characters of $\torus$.
Let $\latticecharacterreal:=\latticecharacter\otimes_{\mathbb{Z}}\mathbb{R}\simeq \euclidspaceN$.
By definition, both $\chowpolytope$ and $\hurwitzpolytope$ are contained in $\latticecharacterreal$.
For an element $\lambda$ in the dual lattice $\latticecharacterdual$ of $\latticecharacter$, take an algebraic one parameter subgroup $\lambda(t)$ in $\torus$.
Let $l_\lambda: \latticecharacter\to\mathbb{R}$ be the integral linear functional corresponding to $\lambda$, i.e., 
$$
	l_\lambda(x)=\langle x, \lambda \rangle.
$$
The following is a corollary of Theorem \ref{thm:paul}.
\begin{theorem}[Theorem B \cite{paul12}]\label{thm:paulfutaki}
	The following asymptotic expansion holds as $|t|\to \infty$$:$
	\begin{equation}\label{eq:paulfutakiexpansion}
		\frac{(n+1)V^2}{n}
		\kenergy(\varphi_{\lambda(t)})
		=
		\paulfutaki(\lambda) \log|t|^2
		+\mathcal{O}(1)
	\end{equation}
	where
	\begin{eqnarray}
		\nonumber
				\paulfutaki(\lambda)
		&:=&
				\bigg(
					\frac{n+1}{n}
				\bigg)
				\chowdegree 
				\min\{
				\langle x, \lambda\rangle
				 \mid\,
					x\in \hurwitzpolytope
				\}	
		\\
		\label{eq:paulfutaki}
		&&
				\qquad\qquad
				-
				\hurwitzdegree
				\min\{
					\langle x, \lambda\rangle
					\mid\,
					x\in \chowpolytope
				\}.
	\end{eqnarray}
\end{theorem}

\subsection{Non-Archimedean $K$-energy}
Following \cite{bhj2017, bhj2019, hisamoto16}, we recall that the slope of $K$-energy  is equal to the  intersection number (non-Archimedean $K$-energy) on the total space of the test configuration of $(X,L)$.
For our purpose, we consider only the compactified test configuration induced by a one parameter subgroup $\lambda(t)$.

An element $\lambda\in\latticecharacterdual$ induces an algebraic one parameter subgroup $\lambda:\ctimes\to \glg$.
We denote the Zariski closure
$$
	\testconfigurationspace
	:=
	\overline{
		\{
			(\lambda(t) x, t)\mid\,
			x\in \iota(X),\, t\in \ctimes
		\}
	}
	\subseteq
	\projectivespace\times \ctimes.
$$ 
Let $\testconfigurationbundle$ be the pull back of $\mathcal{O}(1)_{\projectivespace}$.
Then the pair $(\testconfigurationspace,\testconfigurationbundle)$ constitutes \textit{a test configuration} of $(X,L)$ (\cite{donaldson02}), that is to say, the projection $\pi:\testconfigurationspace\to \ctimes$ is $\ctimes$-equivariant proper flat morsphism such that $(\pi^{-1}(t), \testconfigurationbundle\mid_{\pi^{-1}(t)}) \simeq (X,L)$.
We denote the central fiber by $((\testconfigurationspace)_0,(\testconfigurationbundle)_0)=(\pi^{-1}(0),\testconfigurationbundle\mid_{\pi^{-1}(0)})$.

In \cite{bhj2019}, Boucksom-Hisamoto-Jonsson prove that the slope of $K$-energy along the ray $\varphi_{\lambda(t)}$ is equal to \textit{the non-Archimedean $K$-energy} defined by the intersection number
\begin{equation}\label{eq:nakenergy}
	\nakenergy(\testconfigurationspace,\testconfigurationbundle):
	=
	\frac{1}{V}
	\big(
		K^{\log}_{\bar{\testconfigurationspace}/\mathbb{P}^1}
		\cdot
		\bar{\testconfigurationbundle}^n
	\big)
	+
	\frac{\averagescalar}{V(n+1)}
	\big(
		\bar{\testconfigurationbundle}^{n+1}	
	\big)
\end{equation}
on $(\bar{\testconfigurationspace},\bar{\testconfigurationbundle})$.
In above, $\bar{\pi}:(\bar{\testconfigurationspace},\bar{\testconfigurationbundle})\to \mathbb{P}^1$ denotes the compactification of $(\testconfigurationspace,\testconfigurationbundle)$ in a canonical way (see Definition 2.4 \cite{bhj2017} for the detail) and $K^{\log}_{\bar{\testconfigurationspace}/\mathbb{P}^1}$ denotes the relative logarithmic canonical divisor
\begin{eqnarray*}
			K^{\log}_{\bar{\testconfigurationspace}/\mathbb{P}^1}
	&=&
			K_{\bar{\testconfigurationspace}/\mathbb{P}^1}			-
			\bar{\pi}^*K_{\mathbb{P}^1}
	\\
	&=&
			K_{\bar{\testconfigurationspace}}
			-
			\bar{\pi}^*K_{\mathbb{P}^1}
			+
			(\testconfigurationspace)_{0,\mathrm{red}}
			-(\testconfigurationspace)_0.			
\end{eqnarray*}

\begin{theorem}[Theorem 3.6 \cite{bhj2017}]\label{thm:bhj}
	Let $(\testconfigurationspace,\testconfigurationbundle)$ be the test configuration induced by $\lambda\in\latticecharacterdual$ as before.
	Then we have
	\begin{equation}\label{eq:bhj}
		\lim_{s\to +\infty}
		\frac{\kenergy(\varphi_{\lambda(e^{-s})})}{s}
		=
		2\nakenergy(\testconfigurationspace,\testconfigurationbundle).
	\end{equation}
\end{theorem}
\begin{remark}
	The factor $2$ in the right hand of (\ref{eq:bhj}) does not appear in \cite{bhj2017}.
	It appears due to the difference in the normalizations of the K\"ahler potentials (equivalently, the conformal factor of a Hermitian metric on $L$) between \cite{paul12} and \cite{bhj2017}.
\end{remark}

\subsection{Toric case}
If $(X,L)$ is toric, then $\nakenergy(\testconfigurationspace,\testconfigurationbundle)$ is interpreted as the functional on convex piecewise-linear functions on $Q$.
We refer to Remark 1.2 in \cite{hisamoto16} and  Theorem 5.1 in \cite{delcroix20} for the argument of this subsection, although it would follow straightforwardly from \cite{bhj2017}.

Since $\lambda\in\latticecharacterdual$ and $\torustoric\subset \torus$, the associated one parameter subgroup $\lambda(t)$ is commutative with $\torustoric$.
Then $\torustoric\times \ctimes$ acts on $(\testconfigurationspace,\testconfigurationbundle)$ effectively where the action of the second factor of $\torustoric\times \ctimes$ comes from the action of $\lambda(t)$.
Hence, the compactification $(\bar{\testconfigurationspace},\bar{\testconfigurationbundle)}$ is an $(n+1)$-dimensional polarized toric variety.
Let $\polytopetc\subset \latticecharacterreal\otimes\mathbb{R}=\mathbb{R}^{n+1}$ be the corresponding momentum polytope to $(\bar{\testconfigurationspace},\bar{\testconfigurationbundle)}$.
The polytope $\polytopetc$ is the form of
\begin{equation}\label{eq:polytopetc}
	\{
		(x,h)\in Q\times \mathbb{R}
		\mid\,
		\functionlambda(x)
		\le
		h
		\le 
		c
	\}	
\end{equation}
for some integral convex piecewise-linear function $\functionlambda$ on $Q$ and some constant $c\,(\ge \max_Q \functionlambda(x))$.
The choice of $\functionlambda$ will be discussed later.
Donaldson \cite{donaldson02} introduces the functional $\toricdf$ on the space of piecewise-linear functions on $Q$ defined by
$$
	\toricdf(g):=
	\int_{\partial Q} g d\mu
	- 
	n
	\frac{\volume(\partial Q)}{\volume(Q)}
	\int_Q g dx.
$$
The measure $dx$ denotes the Lebesgue measure and $d\nu$ is the measure on $\partial Q$ so that
$
	dx_1\wedge \cdots \wedge dx_n = \pm d\nu\wedge dh.
$
Here, $h$ is the defining polynomial of a facet of $Q$ which is the form of
$$
	h(x)= \langle x, u \rangle +\mathrm{constant}	
$$
where $u$ is a primitive normal vector of the facet.
Notice that 
$$
	\volume(Q) = n! \int_Q dx, \,\,
	\volume(\partial Q) = (n-1)!\int_{\partial Q} d\nu.
$$
\begin{theorem}
[\cite{donaldson02}, \cite{bhj2017, bhj2019}]
\label{thm:df_nakenergy}
	Let $(\testconfigurationspace,\testconfigurationbundle)$ be the toric test configuration induced by $\lambda\in\latticecharacterdual$ as before.
	Let $\functionlambda$ be the corresponding integral convex piecewise-linear function on $Q$ which defines the momentum polytope $\polytopetc$ of $(\bar{\testconfigurationspace},\bar{\testconfigurationbundle})$.
	Then we have
	\begin{equation}\label{eq:dfnakenergy}
		\nakenergy (\testconfigurationspace,\testconfigurationbundle)
		=
		\frac{n!}{V}\toricdf(\functionlambda).
	\end{equation}
\end{theorem}
\begin{proof}
Proposition 2.8 \cite{bhj2019} (originally Definition 7.13 \cite{bhj2017}) says that the non-Archimedean $K$-energy $\nakenergy (\testconfigurationspace,\testconfigurationbundle)$ coincides with the Donaldson-Futaki invariant $\df (\testconfigurationspace,\testconfigurationbundle)$
defined in \cite{donaldson02} if the central fiber is reduced, i.e., the function $\functionlambda$ takes an integral value at each lattice point in $Q$ and $\partial Q$.
Proposition 7.14 \cite{bhj2017} says that $\nakenergy (\testconfigurationspace,\testconfigurationbundle)$ is homogeneous under the base change $t\mapsto t^d$ of $(\testconfigurationspace,\testconfigurationbundle)$.
Notice that under the base change, the function $\functionlambda$ is changed to $d\functionlambda$ and
$$
	\toricdf(d\functionlambda)= d\toricdf(\functionlambda).
$$ 
Hence, we find that
\begin{eqnarray}
	\nonumber
		\nakenergy (\testconfigurationspace,\testconfigurationbundle)
	&=&
		\frac{1}{d} \nakenergy (\testconfigurationspace',\testconfigurationbundle')
	\\
	\label{eq:donaldsonfutaki}
	&=&
		\frac{1}{dV}\df(\testconfigurationspace',\testconfigurationbundle')	
\end{eqnarray}
where $(\testconfigurationspace',\testconfigurationbundle')$ is obtained by an appropriate base change of $(\testconfigurationspace,\testconfigurationbundle) $ so that its central fiber is reduced.

\begin{remark}
In (\ref{eq:donaldsonfutaki}), we refer to Definition 3.3 in \cite{bhj2017} for the definition of the Donaldson-Futaki invariant that is equal to the original definition in \cite{donaldson02} multiplied by $(-2)$.
\end{remark}

On the other hand, Proposition 4.2.1 \cite{donaldson02} says that 
$$
	\df(\testconfigurationspace',\testconfigurationbundle')
	=
	n!\toricdf(d \functionlambda).
$$
Then we get the desired equality (\ref{eq:dfnakenergy}).
The factor $n!$ in the above comes from the gap between the integration of $g$ on $Q$ (and $\partial Q$) and the corresponding intersection number.
More precisely,
\begin{eqnarray*}
		\frac{\averagescalar}{(n+1)}(\bar{\testconfigurationbundle}^{n+1})
	&=&
		\frac{1}{(n+1)}
		\bigg(
				\frac{n\volume(\partial Q)}{\volume(Q)}
		\bigg)
		(n+1)!
		\int_Q \functionlambda dx
	\\
	&=&
		n!
		\bigg(
				\frac{n\volume(\partial Q)}{\volume(Q)}
		\bigg)
		\int_Q \functionlambda dx.
\end{eqnarray*}
The proof is completed.
\end{proof}

From Theorem \ref{thm:paulfutaki}, Theorem \ref{thm:bhj} and Theorem \ref{thm:df_nakenergy},  we see the following corollary.
\begin{corollary}\label{cor:hurwitzweight}
	For $\lambda\in \latticecharacterdual$, 
	$$
		\paulfutaki(\lambda)
		= 
		-
		n!
		\bigg(
			\frac{n+1}{n}	
		\bigg)
		V\toricdf(\functionlambda)
	$$
	where $\functionlambda$ is the corresponding convex integral piecewise-linear function on $Q$ to the compactified toric test configuration $(\bar{\testconfigurationspace},\bar{\testconfigurationbundle})$.
\end{corollary}
\begin{proof}
From (\ref{eq:paulfutakiexpansion}) and (\ref{eq:bhj}), we have 
$$
	\paulfutaki(\lambda)
	=
	-V^2
	\bigg(
		\frac{n+1}{n}
	\bigg)
	\nakenergy(\testconfigurationspace,\testconfigurationbundle).
$$
From (\ref{eq:dfnakenergy}), we have
$$
	\paulfutaki(\lambda)
	=
	-
	n!
	\bigg(
		\frac{n+1}{n}	
	\bigg)
	V
	\toricdf(\functionlambda),
$$
which is the desired equality.
\end{proof}

\subsection{Chow polytopes and Aubin functional}
A similar result as Corollary \ref{cor:hurwitzweight} holds for the Chow polytope.
\begin{theorem}[\cite{tian94,zhang96, ps03, paul04}]\label{thm:chownorm}
	Let the situation be the same as Theorem \ref{thm:paul}.
	For any $\lambda\in\glg$,
	\begin{equation}\label{eq:chownorm}
		(n+1)V\fzero(\varphi_\lambda)
		=
		-
		\log
		\frac{\|\lambda\cdot \chow\|^2}
		{\|\chow\|^2}.
	\end{equation}
\end{theorem}
\noindent
The norm appeared in (\ref{eq:chownorm}) is the same as the one in (\ref{eq:paul}) introduced by Tian \cite{tian94}.
This norm is also known as \textit{the Chow norm} introduced by Zhang \cite{zhang96}.
\begin{corollary}\label{cor:chowasymptotic}
	For $\lambda\in \latticecharacterdual$, let $\lambda(t)$ be the corresponding one parameter subgroup in $\torus$.
	The following asymptotic expansion holds as $|t|\to 0$:
	\begin{equation}\label{eq:chowasymptotic}
		-(n+1)V\fzero(\varphi_{\lambda(t)})
		=
		\chowfutaki(\lambda)\log|t|^2
		+\mathcal{O}(1)
	\end{equation}
	where
	\begin{equation*}
			\chowfutaki(\lambda):
		=
			\min\{											\langle x, \lambda\rangle
				\mid\,
				x\in \chowpolytope
				\}.
	\end{equation*}	
\end{corollary}
Applying Corollary \ref{cor:chowasymptotic} to the  toric case as Corollary \ref{cor:hurwitzweight}, we have the following proposition.
\begin{proposition}\label{prop:chowweight}
		Let $(X,L)$ be a polarized smooth toric variety as before.
	For $\lambda\in \latticecharacterdual$, we have
	$$
			\min\{
				\langle x, \lambda\rangle
				\mid\,
				x\in \chowpolytope
				\}	
			=
			(n+1)!\toricaubin(\functionlambda -c)
	$$
	where the functional $\toricaubin(g)$ defined in \cite{donaldson02} by
	$$
		\toricaubin(g):=
		\int_Q g dx.
	$$
	Here, the function $\functionlambda$ and the constant $c$ correpond to the compactified test configuration $(\bar{\testconfigurationspace},\bar{\testconfigurationbundle})$ as in $(\ref {eq:polytopetc})$.
\end{proposition}
\begin{proof}
Theorem 3.6 \cite{bhj2017} shows the expansion
$$
	-\fzero(\varphi_{\lambda(t)})
	= 
	-\namaenergy(\testconfigurationspace,\testconfigurationbundle)
	\log|t|^2
	+\mathcal{O}(1),
$$
where $\namaenergy (\testconfigurationspace,\testconfigurationbundle) $ is \textit{the non-Archimedean Monge-Amp\`ere energy} with respect to the test configuration $(\testconfigurationspace,\testconfigurationbundle)$ defined by
$$
	\namaenergy(\testconfigurationspace,\testconfigurationbundle)
	=
	\frac{(\bar{\testconfigurationbundle})^{n+1}}{(n+1)V}.
$$
From (\ref{eq:chowasymptotic}), we have
$$
	\chowfutaki(\lambda)
	=
	-(\bar{\testconfigurationbundle})^{n+1}
	=
	(n+1)!\int_Q (\functionlambda -c) dx.
$$
The proof is completed.
\end{proof}

\section{Proof}\label{sec:proof}
We prove Theorem \ref{thm:chowpolytope} and Theorem \ref{thm:main}.
Let 
$$
	A = 
	\{
		\omega_0,\ldots,\omega_N
	\}
	\subset
	\euclidspacen	
$$
be the set of lattice points in the momentum polytope $Q$ of $(X,L)$ and its boundary $\partial Q$.
Take an element 
$$
	\lambda
	=
	(\lambda_0,\ldots,\lambda_N)	
	\in\latticecharacterdual.
$$
We assume that 
$$
	\max_{k}\lambda_k = 0.
$$
Remark that the invariant $\paulfutaki(\lambda)$ is unchanged under the addition of constants to $\lambda$.

The corresponding one parameter subgroup of $\glg$ is represented by
$$
	\lambda(t)
	=
	\begin{pmatrix}
			t^{\lambda_0} &&
		\\
			& \ddots &
		\\
			&& t^{\lambda_N }
	\end{pmatrix}
	\in \torus \subset \glg.
$$
The linear functional $l_\lambda$ is defined by
$$
	l_\lambda(x):= \langle x, \lambda\rangle,
	\quad
	x\in \latticecharacterreal.
$$
Let $\Afunction$ be the space of functions $g:A\to\mathbb{R}$.
We define the pairing of $g\in\Afunction$ and its dual $x=(x_0,\ldots,x_N)\in\latticecharacter$ by
$$
	(x,g):= \sum_{k=0}^{N} x_k\cdot g(\omega_k).
$$

\subsection{Polytope of the test configuration}\label{subsec:polytopetestconfiguration}
For $0\le k \le N$, we define
$$
	\pointheight_k:=(\omega_k,\lambda_k)
	\in \mathbb{R}^n\times \mathbb{R}.
$$
For the monomial
$$
	z^{\omega_k}:=\prod_{i=1}^n z_i^{\omega_{k,i}}
$$
where
$\omega_k=(\omega_{k,1},\ldots,\omega_{k,n})$, the $\ctimes$-action of $(\testconfigurationspace,\testconfigurationbundle)$ implies that
\begin{equation}\label{eq:laurent}
	t\cdot z^{\omega_k}
	= 
	t^{\lambda_k}z^{\omega_k}
	=
	\widetilde{z}^{\pointheight_k}	
\end{equation}
where $\widetilde{z}=(z_1,\ldots,z_n,t)$.
Then the polarized toric variety $(\testconfigurationspace,\,\testconfigurationbundle)$ 
corresponds to the momentum polytope defined by the convex hull of
$$
	\{
		(\omega_k,h)\in \latticecharacter \times \mathbb{Z}
		\mid
		\,
		0\le k\le N,\,
		\lambda_k\le h
	\}.
$$
Moreover, its compactification $(\bar{\testconfigurationspace},\,\bar{\testconfigurationbundle})$ has the momentum polytope $\polytopetc\in \latticecharacterreal\times \mathbb{R}$ that is the convex hull of
$$
		\big\{
			\pointheight_k
		\big\}_{0\le k \le N}
		\,
		\bigcup
		\,
		\big\{
			(\omega_k,0)
		\big\}_{0\le k \le N}.
$$
In the above, we let the constant $c$ in (\ref{eq:polytopetc}) be equal to zero.
The polytope $\polytopetc$ defines a convex integral piecewise-linear function $\functionlambda$ on $Q$ satisfying (\ref{eq:polytopetc}).
If $\pointheight_k$ is a vertex of $\polytopetc$, then $\functionlambda(\omega_k)$ is equal to $\lambda_k$.
Otherwise, $\functionlambda(\omega_k)$ is more than or equal to $\lambda_k$.
Moreover, the vertical projection of the polytope $\polytopetc$ defines a subdivision (possibly not triangulation).
We denote it by $\triangulationlambda$.

\subsection{Cones in $\latticecharacterreal$}
For a triangulation $T$ and an element $\lambda\in\latticecharacterreal$, let $g_{\lambda,T} $ be the $T$-piecewise-linear function on $Q$ defined by
$$
	g_{\lambda,T}(\omega_k)=\lambda_k
$$
for $\omega_k\in \Sigma_T(0)$.

For a regular triangulation $T$ of $(Q,A)$, we define  the cone $\conet$ by the set of $\lambda\in\latticecharacterreal$ so that $g_{\lambda,T}$ is convex and 
$$
	g_{\lambda,T}(\omega_k)<\lambda_k
$$
for each $\omega_k\not\in\Sigma_T(0)$.
In particular, $\lambda\in\conet\cap\mathbb{Z}^{N+1}$ if and only if the one parameter subgroup $\lambda(t)$ induces the test configuration $(\testconfigurationspace,\testconfigurationbundle)$ which the vertical projection of $\polytopetc$ provides the triangulation $T$ of $(Q,A)$.
The cone $\conet$ has the maximal dimension.

For each vertex $\eta$ of $\chowpolytope$, we define the cone $\normalcone(\eta)$ by the normal cone of $\chowpolytope$ at $\eta$.
In other words, $\lambda\in \normalcone(\eta) \subset \latticecharacterreal$ if and only if the linear function $l_\lambda: \chowpolytope\to\mathbb{R}$ has the unique maximum at $\eta$.
The cone $\normalcone(\eta)$ also has the maximal dimension.

\subsection{Proof of Theorem \ref{thm:chowpolytope}}
\label{subsec:proofchowpolytope}

\begin{proposition}[Lemma 1.8, Chapter 7 \cite{gkz94}]\label{prop:chowweightandtoricaubin}
	Let $(X,L)$ be a smooth polarized toric variety with the momentum polytope $Q$.
	Let $A$ be the set of all lattice points on $Q\cup \partial Q$.
	Let $T$ be any triangulation of $(Q,A)$. 
	Let $g$ be any piecewise-linear function with respect to $T$.
	Assume that the corresponding constant $c$ to $T$ is zero.
	Then, we have
	\begin{equation}\label{eq:chowweightandtoricaubin}
		(\eta_T, g)
		=
		(n+1)! \toricaubin(g).
	\end{equation}
\end{proposition}

\begin{proof}
The proof is the same as the proof of Lemma 1.8, Chapter 7 \cite{gkz94}.
For any $\sigma\in \Sigma_T(n)$, 
\begin{equation}\label{eq:integral_pairing}
	\int_\sigma g dx
	=
	\bigg(
		\frac{1}{n+1}\sum_{j=0}^n g(\omega_{k_j})
	\bigg)
	\bigg(
		\int_\sigma dx
	\bigg)
	=
	\frac{\volume(\sigma)}{(n+1)!}\sum_{j=0}^n g(\omega_{k_j})	
\end{equation}
where $\{\omega_{k_0},\ldots,\omega_{k_n}\}$ be the set of the vertices of $\sigma$.
By definition of $\eta_T(\omega_k)$, the above equality (\ref{eq:integral_pairing}) implies (\ref{eq:chowweightandtoricaubin}).
\end{proof}

\begin{corollary}\label{cor:chowpairing}
	Let $(X,L)$ and $Q$ be same as Proposition $\ref{prop:chowweightandtoricaubin}$.
	Assume that the vertical projection of $\polytopetc$ into $\latticecharacterreal\simeq \euclidspacen$ induces a triangulation $\triangulationlambda$ of $(Q,A)$.
	Then we have
	\begin{equation}\label{eq:chowpairing}
		\max\{
			\langle x, -\lambda \rangle
			\mid\,
			x \in \chowpolytope
		\}
		=
		\langle \eta_{\triangulationlambda}, -\lambda \rangle.
	\end{equation}
\end{corollary}
\begin{proof}
Let $\functionlambda$ be the convex piecewise-linear function on $Q$ defined by (\ref{eq:polytopetc}).
If $\pointheight_k$ is not any vertex of $\polytopetc$, then $\omega_k\not\in\Sigma_{\triangulationlambda}(0)$.
This implies that the $k$-th element of $\eta_{\triangulationlambda}$ is equal to zero.
Hence, we have
$$
	(\eta_{\triangulationlambda},\functionlambda)=\langle \eta_{\triangulationlambda}, \lambda\rangle.
$$ 
The above equality with Proposition \ref{prop:chowweight} and Proposition \ref{prop:chowweightandtoricaubin} implies
$$
	\min\{
		\langle x, \lambda \rangle
		\mid\,
		x \in \chowpolytope
	\}
	=
	\langle \eta_{\triangulationlambda}, \lambda \rangle	
$$
that is equivalent to (\ref{eq:chowpairing}).	
\end{proof}

Now, we prove that for any regular triangulation $T$ of $(Q,A)$, the vertex $\eta_T$ of $\secondarypolytope(X)$ is a vertex of $\chowpolytope$.
Take any regular triangulation $T$ of $(Q,A)$.
Since $\normalcone(\eta)$ has the maximal dimension and
$$
	\bigcup_{\eta} \normalcone(\eta)
	=\latticecharacterreal
$$
where $\eta$ runs through all vertices of $\chowpolytope$, there exists some vertex $\eta$ of $\chowpolytope$ such that
\begin{equation}\label{eq:nonempty}
		\interior{(\conet)}
		\bigcap
		\interior(\normalcone(\eta))
		\neq \emptyset.
\end{equation}
For any $-\lambda\in 	\interior(\conet)\cap\interior(\normalcone(\eta))$, the vertex $\eta$ is the unique maximizer of the restricted linear function $l_{(-\lambda)}:\chowpolytope\to\mathbb{R}$.
Corollary \ref{cor:chowpairing} implies
\begin{equation*}
	\langle \eta, \lambda\rangle 
	=
	\langle\eta_T, \lambda\rangle.
\end{equation*}
The generality of the choice of $\lambda$ implies that $\eta=\eta_T$.
Hence, $\eta_T$ is the vertex $\eta$ of $\chowpolytope$.

Next, we prove the converse.
Take any vertex $\eta$ of $\chowpolytope$. 
As before, there exists a triangulation $T$ so that (\ref{eq:nonempty}) holds because the union
$$
	\bigcup_{T} \conet =\latticecharacterreal
$$
where $T$ runs through all regular triangulations of $(Q,A)$.
Take $-\lambda\in \conet \cap \mathbb{Z}^{N+1}$.
Corollary \ref{cor:chowpairing} implies
$$
	\langle \eta, \lambda\rangle 
	=
	\langle\eta_{T}, \lambda\rangle.
$$
The uniqueness of the maximizer of $l_{\lambda}$ implies that $\eta=\eta_{T}$ because $\eta_T$ is a vertex of $\chowpolytope$ proved above.
Therefore, the proof of Theorem \ref{thm:chowpolytope} is completed.

\subsection{Proof of Theorem \ref{thm:main}}
The proof is totally the same as Subsection \ref{subsec:proofchowpolytope} after replacing Proposition \ref{prop:chowweightandtoricaubin} and Corollary \ref{cor:chowpairing} by the followings respectively.
\begin{proposition}[Proposition 6.3 \cite{os22}]\label{prop:hurwitzweightandtoricdf}
	Let the situation be the same as Proposition $\ref{prop:chowweightandtoricaubin}$.
	Then we have
	\begin{eqnarray}
		\nonumber
		&&
		(n+1)!\volume(Q)\toricdf(g)
		\\
		\label{eq:toricdfhurwitzvector}
		&&
		\qquad
		=
		\big( 
			n\hurwitzdegree\eta_{T}
			-
			(n+1)\chowdegree\xi_T,
			g
		\big).
	\end{eqnarray}
\end{proposition}
\begin{proof}
Although the proof is the same as the proof of  Proposition 6.3 \cite{os22}, we write it for the readers.
Recall that
$$
	\chowdegree= \degree = \volume(Q),
$$
and
\begin{eqnarray*}
		\hurwitzdegree
	&=&
		(n+1)\degree - \frac{\degree}{n}\averagescalar
	\\
	&=&
		(n+1)\volume(Q) 
		-
		\frac{\volume(Q)}{n} 
		\cdot
		\bigg(
			n \frac{\volume(\partial Q)}{\volume(Q)}
		\bigg)
	\\
	&=&
		(n+1)\volume(Q)-\volume(\partial Q).
\end{eqnarray*}
The latter follows from the formula (5.53) in \cite{paul12}.
As before, we have
$$
	(\eta_T, g)
	= (n+1)! \int_Q g dx.
$$
Similarly, we have
$$
	(\eta_{T,n-1}, g)
	= n! \int_{\partial Q} g d\nu.
$$
The first term in the right hand of (\ref{eq:toricdfhurwitzvector}) is equal to
\begin{eqnarray*}
		n\hurwitzdegree(\eta_T, g)
	&=&
		n(n+1)(n+1)!
		\volume(Q)
		\int_Q g dx
	\\
	&&
		\qquad
		-n(n+1)!
		\volume(\partial Q)
		\int_Q g dx.
\end{eqnarray*} 
The second term in the right hand of (\ref{eq:toricdfhurwitzvector}) is equal to
\begin{eqnarray*}
		(n+1)\chowdegree(\xi_T,g)
	&=&
		(n+1)\volume(Q)(n\eta_T,g)
	\\
	&&
		\qquad
		-
		(n+1)\volume(Q)(\eta_{T,n-1},g)
	\\
	&=&
		n(n+1)(n+1)!
		\volume(Q)
		\int_Q g dx
	\\
	&&
		\qquad
		- (n+1)!\volume(Q)\int_{\partial Q} g d\mu.
\end{eqnarray*} 
Hence, the right hand of (\ref{eq:toricdfhurwitzvector})  is equal to
\begin{eqnarray*}
		&&
			(n+1)!\volume(Q)\int_{\partial Q} g d\mu
			- n(n+1)!\volume(\partial Q)\int_Q g dx
		\\
		&&
			\qquad
			=
			(n+1)!\volume(Q)
			\bigg(
				\int_{\partial Q}gd\mu
				-
				n \frac{\volume(\partial Q)}{\volume(Q)}
				\int_Q g dx
			\bigg)
		\\
		&&
			\qquad
			=
			(n+1)!\volume(Q)\toricdf(g).
\end{eqnarray*}
The proof is completed.
\end{proof}

\begin{corollary}\label{cor:hurwitzpairing}
	Let the situation be same as Corollary \ref{cor:chowpairing}.
	Then we have
	\begin{equation*}\label{eq:hurwitzpairing}
		\max\{
			\langle x, -\lambda \rangle
			\mid\,
			x \in \hurwitzpolytope
		\}
		=
		\langle \xi_{\triangulationlambda}, -\lambda \rangle.
	\end{equation*}
\end{corollary}
\begin{proof}
Recall 
$$
	\langle \eta_{\triangulationlambda}, \lambda \rangle
	=
	(\eta_{\triangulationlambda}, g_\lambda),
	\quad
	\langle \xi_{\triangulationlambda}, \lambda \rangle
	=
	(\xi_{\triangulationlambda}, g_{\lambda}).
$$
Proposition \ref{prop:hurwitzweightandtoricdf} implies that
\begin{eqnarray*}
		\langle \xi_{\triangulationlambda}, \lambda \rangle
	&=&
		\bigg(
			\frac{n}{n+1}
		\bigg)
		\bigg(
			\frac{\hurwitzdegree}{\chowdegree} 
		\bigg)
		(\eta_{\triangulationlambda}, g_\lambda)
		-
		\frac{n!\volume(Q)}{\chowdegree}
		\toricdf(g_\lambda)
	\\
	&=&
		\bigg(
			\frac{n}{n+1}
		\bigg)
		\bigg(
			\frac{\hurwitzdegree}{\chowdegree} 
		\bigg)
		(\eta_{\triangulationlambda}, g_\lambda)
		-
		n!\toricdf(g_\lambda).
\end{eqnarray*}
Theorem \ref{thm:chowpolytope}, Corollary \ref{cor:hurwitzweight} and (\ref{eq:paulfutaki}) imply that
\begin{eqnarray*}
		-n!\toricdf(g_\lambda)
	&=&
		\frac{1}{V}
		\bigg(
			\frac{n}{n+1}
		\bigg)
		\paulfutaki(\lambda)
	\\
	&=&
		\min\{
			\langle x, \lambda \rangle
			\mid\,
			x \in \hurwitzpolytope
		\}		
	\\
	&&
		\,
		-
		\bigg(
			\frac{n}{n+1}
		\bigg)
		\bigg(
			\frac{\hurwitzdegree)}{\chowdegree)} 
		\bigg)
		\min\{
			\langle x, \lambda \rangle
			\mid\,
			x \in \chowpolytope
		\}	
	\\
	&=&
		\min\{
			\langle x, \lambda \rangle
			\mid\,
			x \in \hurwitzpolytope
		\}		\\
	&&
		\qquad
		-
		\bigg(
			\frac{n}{n+1}
		\bigg)
		\bigg(
			\frac{\hurwitzdegree)}{\chowdegree)} 
		\bigg)
		(\eta_{\triangulationlambda}, \functionlambda).
\end{eqnarray*}
where
$
	V= \volume(Q).
$
Hence, we have
$$
	\langle \xi_{\triangulationlambda}, \lambda \rangle
	=
	\min\{
			\langle x, \lambda \rangle
			\mid\,
			x \in \hurwitzpolytope
	\}.
$$
The proof is completed.
\end{proof}
Replacing Corollary \ref{cor:chowpairing} (resp. $\chowpolytope$ in the definition of $\normalcone(\eta)$) by Corollary \ref{cor:hurwitzpairing} (resp. $\hurwitzpolytope$), the same argument of the proof of Theorem \ref{thm:chowpolytope} completes the proof of Theorem \ref{thm:main}.

\end{document}